\documentclass[12pt,psamsfonts]{amsart}
\usepackage{amsmath,amsthm,amsfonts,amssymb}
\usepackage{eucal}
\usepackage{graphicx}
\usepackage[all,knot]{xy}
\xyoption{arc}

\newcommand{\D}{\mathcal{D}}

\newcommand{\CC}{\mathcal{C}}

\newcommand{\one}{\mathbf{1}}

\newcommand{\Z}{\mathbb Z}

\newcommand{\comments}[1]{}

\renewcommand{\one}{\mathbf{1}}

\renewcommand{\CC}{\mathcal{C}}
\renewcommand{\D}{\mathcal{D}}
\renewcommand{\Z}{\mathbb{Z}}

\numberwithin{equation}{section}

\newtheorem{theorem}{Theorem}[section]

\newtheorem{lem}[theorem]{Lemma}

\newtheorem{prop}[theorem]{Proposition}
\theoremstyle{definition}

\newtheorem{definition}[theorem]{Definition}

\begin{document}
\title[]{Classification of Metaplectic Modular Categories}

\author{Eddy Ardonne}
\email{ardonne@fysik.su.se}
\address{Department of Physics\\Stockholm University
    \\Albanova University Center
    \\SE-106 91 Stockholm
    Sweden}

\author{Meng Cheng}
\email{mcheng@microsoft.com}
\address{Microsoft Station Q\\
    University of California\\
    Santa Barbara, CA 93106-6105\\
    U.S.A.}

\author{Eric C. Rowell}
\email{rowell@math.tamu.edu}
\address{Department of Mathematics\\
    Texas A\&M University \\
    College Station, TX 77843\\
    U.S.A.}

\author{Zhenghan Wang}
\email{zhenghwa@microsoft.com}
\address{Microsoft Station Q and Dept of Mathematics\\
    University of California\\
    Santa Barbara, CA 93106-6105\\
    U.S.A.}

\thanks{The first author is supported in part by the Swedish research council. The third and fourth authors are partially supported by NSF grants DMS-1410144 and DMS-1411212 respectively.}

\begin{abstract}

We obtain a classification of metaplectic modular categories: every metaplectic modular category is a gauging of the particle-hole symmetry of a cyclic modular category.  Our classification suggests a conjecture that every weakly-integral modular category can be obtained by gauging a symmetry of a pointed modular category.

\end{abstract}
\maketitle

\section{Introduction}

Achieving a classification of modular categories analogous to the classification of finite abelian groups is an interesting mathematical problem \cite{BNRW1,BNRW2}.  In this note, we classify metaplectic modular categories.  Our classification suggests a close connection between finite abelian groups and weakly integral modular categories via gauging, thus leads to a potential approach to proving the Property F conjecture for weakly integral modular categories \cite{BBCW,CGPW,NR}.

A simple object $X$ is weakly-integral if its squared quantum dimension $d_X^2$ is an integer.  A modular category is weakly-integral if every simple object is weakly-integral.
Inspired by the applications to physics and topological quantum computation, we focus on weakly-integral modular categories \cite{CHW,CW}.  An important class of weakly-integral modular categories is the class of metaplectic modular categories---unitary modular categories with the fusion rules of $SO(N)_2$ for some odd integer $N>1$ \cite{HNW1,HNW2}.  The metaplectic modular categories first appeared in the study of parafermion zero modes, which generalize the Majorana zero modes.  The name {\it metaplectic} comes from the fact that the resulting braid group representations from the generating simple objects in $SO(N)_2$ are the metaplectic representations, which are the symplectic analogues of the spinor representations.  Our main result is a classification of metaplectic modular categories: every metaplectic modular category is a gauging of the particle-hole symmetry of a cyclic modular category.

The property F conjecture says that all braid group representations afforded by a weakly-integral simple object have finite images.  For $SO(N)_2$, the property F conjecture follows from \cite{RW}. It is possible that all weakly-integral modular categories can be obtained from gauging of pointed modular categories---modular categories with all simple objects having their quantum dimension equal to $1$.  Our classification supports this possibility.  If this is true, then a potential approach to the property F conjecture for all weakly-integral modular categories would be to prove that gauging preserves property F.

\section{Cyclic modular categories}

\begin{definition}

Let $\Z_n$ be the cyclic group of $n$ elements.  A $\Z_n$-cyclic modular category is a modular category whose fusion rule is the same as the cyclic group $\Z_n$ for some integer $n$.

\end{definition}

A $\Z_n$-cyclic modular category is determined by a non-degenerate quadratic form $q: \Z_n \rightarrow \mathbb{U}(1)$.  We will denote the $\Z_n$-cyclic modular category determined by such a quadratic form $q$ as $\CC(\Z_n,k)$  for $q(j)=e^{2\pi i s_j}, s_j=\frac{k j^2}{n}, 0\leq j\leq n-1, (k,n)=1$.

First for $M,N$ relatively prime, $\CC(\mathbb{Z}_{MN},k)$ is a direct product of $\CC(\mathbb{Z}_M,kN)$ and $\CC(\mathbb{Z}_N,kM)$. The simple object types, $j$, of $\CC(\Z_{MN},k)$ can be labelled by pairs $(a,b)$, where $j=aM+bN$ and $0\leq a\leq N-1, 0\leq b\leq M-1$. The fusion rules are
\begin{equation}
    j_1\times j_2=(a_1,b_1)\times (a_2,b_2)=([a_1+a_2]_N, [b_1+b_2]_M),
    \label{}
\end{equation}
and the topological twists are $\theta_j=e^{2\pi i s_j}$:
\begin{equation}
    s_j=\frac{kj^2}{MN}=\frac{k(aM+bN)^2}{MN}=\frac{kMa^2}{N}+\frac{kNb^2}{M}+2abk.
    \label{}
\end{equation}

Therefore, we have shown that $\CC(\mathbb{Z}_{MN},k)=\CC(\mathbb{Z}_M,kN)\boxtimes\CC(\mathbb{Z}_N,kM)$.

Next we find all distinct $\mathbb{Z}_{p^a}$-cyclic modular categories, where $p$ is an odd prime.

For $\CC(\mathbb{Z}_{p^a},k)$, write $k=p^l m$, where $p\nmid m$. Note that if $l\geq 1$, the resulting category is not modular (since the form $q(x)=e^{2\pi i kx^2/p^a}$ is degenerate). Therefore, we must assume $(k,p)=1$.
The twist of the $j$-th simple object is $e^{\frac{2\pi ik}{p^{a}}j^2}$. If for $n_1$ and $n_2$, the categories are isomorphic, it means that one can solve the congruent equation
\begin{equation}
    \frac{n_1}{p^{a}}\equiv \frac{n_2j^2}{p^{a}} (\text{mod }1),
    \label{}
\end{equation}
for some $j$ such that $p\nmid j$ (so that $j$ is a generator of $\mathbb{Z}_{p^{a}}$). We need to solve $j^2\equiv n_2^{-1}n_1\,(\text{mod }p^{a})$, which  is solvable when $\left( \frac{n_1}{p^a} \right)=\left( \frac{n_2}{p^a} \right)$. Therefore, there are two distinct theories.

Braided tensor auto-equivalences of the $\Z_n$-cyclic-modular categories are group isomorphisms of $\Z_n$ which preserve the topological twists.  The {\bf particle-hole symmetry} of a $\Z_n$-cyclic modular category with $n$ odd is the $\Z_2$-braided tensor auto-equivalence that maps $j$ to $n-j$.

\section{Metaplectic modular categories}

The unitary modular categories $SO(N)_2$ for odd $N>1$ has $2$ simple objects $X_1, X_2$ of dimension $\sqrt{N}$, two simple objects $\one, Z$ of dimension $1$, and $\frac{N-1}{2}$ objects $Y_i$, $i=1,\ldots,\frac{N-1}{2}$ of dimension $2$.  The fusion rules are:
\begin{enumerate}
 \item $Z\otimes Y_i\cong Y_i$, $Z\otimes X_i\cong X_{i+1}$ (modulo $2$), $Z^{\otimes 2}\cong\one$,
 \item $X_i^{\otimes 2}\cong \one\oplus \bigoplus_{i} Y_i$,
 \item $X_1\otimes X_2\cong Z\oplus\bigoplus_{i} Y_i$,
 \item $Y_i\otimes Y_j\cong Y_{\min\{i+j,N-i-j\}}\oplus Y_{|i-j|}$, for $i\neq j$ and $Y_i^{\otimes 2}=\one\oplus Z\oplus Y_{\min\{2i,N-2i\}}$.
\end{enumerate}
The fusion rules for the subcategory generated by $Y_1$ (with simple objects $\one, Z$ and all $Y_i$) are precisely those of the dihedral group of order $2N$.

\begin{definition}

A metaplectic modular category is a unitary modular category $\CC$ with the same fusion rules as $SO(N)_2$ for some odd $N>1$.

\end{definition}

\begin{theorem}\label{main}

\begin{enumerate}

\item Suppose $\CC$ is a metaplectic modular category with fusion rules $SO(N)_2$, then $\CC$ is a gauging of the particle-hole symmetry of a $\Z_N$-cyclic modular category.
\item For $N=p_1^{\alpha_1}\cdots p_s^{\alpha_s}$ with distinct odd primes $p_i$, there are exactly $2^{s+1}$ many inequivalent metaplectic modular categories.

\end{enumerate}

\end{theorem}

To prove the theorem, we start with two lemmas.

\begin{lem}\label{boson}
The object $Z$ is a boson: $\theta_Z=1$.
\end{lem}

\begin{proof}
Let $Y$ be any of the $\frac{N-1}{2}$ simple objects of dimension $2$.  By orthogonality of the rows of the $S$-matrix, we find that $S_{YZ}=2$.
Observing that $Y\otimes Z\cong Y$, we apply the balancing equation (see e.g. \cite{BK}):
$$S_{ij}\theta_i\theta_j=\sum_{k=0}^{r-1} N_{i^*j}^kd_k\theta_k$$
to obtain: $2\theta_{Y}\theta_Z=S_{YZ}\theta_Y\theta_Z=\theta_Yd_Y=2\theta_Y$.  It follows that $\theta_Z=1$.
\end{proof}

Since $Z$ is a boson (i.e. $\dim(Z)=1$ and $\theta_Z=1$), we may condense (\lq\lq de-equivariantize'' in the categorical language) to obtain a $\Z_2$-graded category \cite{DGNO}.  Since $Z$ interchanges $\one\leftrightarrow Z$ and $X_1\leftrightarrow X_2$ and fixes the $Y_i$ the resulting condensed category $\D:=\CC^{\mathbb{Z}_2}$ has $N$ objects of quantum dimension $1$ in the identity sector $\D_0$ and one object of dimension $\sqrt{N}$ in the non-trivial sector $\D_1$ (see \cite{BBCW}).  Clearly, the fusion rules of $\D_0$ must be identical to those of some abelian group $A$ of order $N$.  In the following, we show that $A\cong\Z_N$.  
As an aside, we point out that the category $\D$ is a Tambara-Yamagami category \cite{TY}.
\begin{lem}\label{fusion}
The fusion rules of $\D_0$ are the same as $\Z_N$.
\end{lem}

\begin{proof}
It is enough to find a tensor generator for $\D_0$, that is, a simple object $U$ so that $\{U^{\otimes i}:i\geq 0\}$ contains all simple objects in $\D_0$.
Now under condensation each object $Y_i$ becomes the sum of two invertible simple objects in $\D_0$.  The image of $Y_i$ under condensation is $Y_i^1\oplus Y_i^2$, a sum of invertible simple objects in $\D_0$.  We denote by $\one_0$ the image of $\one$ and $Z$ under condensation (i.e. the unit object in $\D_0$).  We will proceed to show that $Y_1^1$ is a tensor generator for $\D_0$.

From $Y_1^{\otimes 2}\cong \one\oplus Z\oplus Y_2$ we obtain $$(Y_1^1)^{\otimes 2}\oplus (Y_1^2)^{\otimes 2}\oplus 2Y_1^1\otimes Y_1^2= 2\one_0\oplus Y_2^1\oplus Y_2^2.$$  This implies ${Y_1^1}^*=Y_1^2$, so that $Y_1^2$ appears as some tensor power of $Y_1^1$.  Thus  $Y_1^1$ is a tensor generator provided each $Y_i^{(j)}$ appears in some tensor power of $(Y_1^1\oplus Y_1^2)$.  Since every $Y_i$ appears in some tensor power of $Y_1$ the result follows.
\end{proof}

{\bf Proof of Theorem \ref{main}}:  (1) By Lemmas \ref{boson},\ref{fusion}, each metaplectic modular category is obtained from gauging a $\Z_2$ symmetry of a cyclic modular category. But the particle-hole symmetry is the only non-trivial $\Z_2$ symmetry of a cyclic modular category. (2): As discussed above, there are exactly two cyclic modular categories for each prime power factor in $N$.  When gauging the particle-hole symmetry, there is an additional choice parameterized by $H^3(\Z_2;U(1))\cong \Z_2$ \cite{ENO,BBCW,CGPW}.  Therefore, the number of metaplectic modular categories is $2^{s+1}$.

\section{Witt classes and open problems}
Gauging preserves Witt classes \cite{CGPW}.  Therefore, the Witt classes of metaplectic modular categories are the same as those of the corresponding cyclic modular categories.

\begin{prop}
	The modular category $\CC(\mathbb{Z}_{p^{2a}}, q)$ is a quantum double $\mathcal{Z}(\mathrm{Vec}_{\mathbb{Z}_{p^a}}^\omega)$.
\end{prop}

\begin{proof}
It is easy to see that regardless of the quadratic form $q$, the simple objects $[np^a]$ are all bosons, for $n=0, 1, \dots, p^a-1$. They form a $\mathbb{Z}_{p^a}$ fusion category. In fact, one can define a Lagrangian subalgebra $\bigoplus_{n=0}^{p^a-1}[np^a]$ of $\CC(\mathbb{Z}_{p^{2a}}, q)$. This shows that  $\CC(\mathbb{Z}_{p^{2a}}, q)$ is indeed a quantum double. Now let us condense this subalgebra, which identifies $[j]$ with $[j+np^a]$. Therefore, one can label the distinct simple objects after condensation by $[j], j=0, 1, \dots, p^a-1$. Hence $\CC(\mathbb{Z}_{p^{2a}}, q)$ must be a quantum double of $\mathbb{Z}_{p^a}$, generally twisted by a class in $H^3$ \cite{DGNO}.
\end{proof}

One open question is to prove property F for all metaplectic modular categories.  Another one is to construct universal computing models from metaplectic modular categories by supplementing braidings with measurements \cite{CW}.

\end{document}